\documentclass[11pt,a4paper]{article}

\author{Tom De Medts\footnote{tom.demedts@ugent.be} \\ Michiel Van Couwenberghe\footnote{michiel.vancouwenberghe@ugent.be [PhD Fellow of the Research Foundation - Flanders (Belgium) (F.W.O.-Vlaanderen)]}}
\title{Modules over axial algebras}

\usepackage{amsmath, amsthm, amssymb}
\usepackage{mathtools}
\usepackage{url}
\usepackage[shortlabels]{enumitem}
\setlist{itemsep=.2ex}
\usepackage{geometry}
\usepackage{microtype}
\usepackage{tikz}
\usetikzlibrary{calc,intersections}
\usepackage[textsize=footnotesize,textwidth=17ex]{todonotes}
\usepackage{bbm}

\theoremstyle{plain}
\newtheorem{thm}{Theorem}[section]
\newtheorem{prop}[thm]{Proposition}
\newtheorem{lem}[thm]{Lemma}
\newtheorem{cor}[thm]{Corollary}
\theoremstyle{definition}
\newtheorem{defn}[thm]{Definition}
\newtheorem{example}[thm]{Example}
\newtheorem{rem}[thm]{Remark}

\DeclareMathOperator\aut{Aut}
\DeclareMathOperator\GL{GL}
\DeclareMathOperator\U{U}
\DeclareMathOperator\Char{char}
\DeclareMathOperator\order{o}
\DeclareMathOperator\Z{Z}
\DeclareMathOperator\Miy{Miy}

\newcommand{\mP}{\mathcal{P}}
\newcommand{\mL}{\mathcal{L}}
\newcommand{\mG}{\mathcal{G}}
\newcommand{\mA}{\mathcal{A}}

\newcommand{\dash}{\nobreakdash-\hspace{0pt}}

\begin{document}

\maketitle

\begin{abstract}
	We introduce axial representations and modules over axial algebras as new tools to study axial algebras. All known interesting examples of axial algebras fall into this setting, in particular the Griess algebra whose automorphism group is the Monster group. Our results become especially interesting for Matsuo algebras. We vitalize the connection between Matsuo algebras and 3\dash transposition groups by relating modules over Matsuo algebras with representations of 3\dash transposition groups. As a by-product, we define, given a Fischer space, a group that can fulfill the role of a universal 3\dash transposition group.
\end{abstract}

\begin{small}
\end{small}

\section{Introduction}

In 1982, Robert Griess constructed a commutative non-associative algebra, called the Griess algebra, in order to prove the existence of the Monster group \cite{Gri82}. Igor Frenkel, James Lepowsky and Arne Meurman observed that (a deformation of) this Griess algebra can be retrieved as a weight\dash 2 component of a \emph{vertex operator algebra (VOA)} $V^\natural$ \cite{FLM88}.
Richards Borcherds then obtained a proof of the moonshine conjectures by studying this VOA \cite{Bor86}.

The weight-2 components of VOAs similar to $V^\natural$ give rise to algebras with properties like the ones for the Griess algebra. Alexander Ivanov defined \emph{Majorana algebras} in an attempt to axiomatize such algebras. The definition of an \emph{axial algebra} was introduced by Jonathan Hall, Sergey Shpectorov and Felix Rehren in 2015 by stripping away some of the requirements in the definition of Majorana algebras in such a way that many results concerning Majorana algebras still remain true.

Since axial algebras have only recently been defined, this new research subject is very much unexplored
and it is not even clear whether the definition of an axial algebra is yet in its final stage.

\medskip

In this paper, we provide some new tools to study axial algebras.
First, we define \emph{axial representations} of a group as a generalization of Majorana representations which were introduced in \cite{IPSS10}. This notion is perhaps more foundational than the definition of an axial algebra itself, since it describes the connection between axial algebras and groups,
very much like the connection between the Griess algebra and the Monster group.

Second, we present a natural definition of \emph{modules over axial algebras}. If an axial algebra $A$ is an axial representation of a group $G$, each $A$\dash module gives rise to a $\U(A)$\dash module where $\U(A)$ is a certain central extension of $G$; see Theorem \ref{thm:moduletorep} below.
In this sense, our definition of modules over axial algebras reinforces the connection between axial algebras and groups.

Third, we investigate the theory of modules more thoroughly for \emph{Matsuo algebras} over \emph{Fischer spaces}, an important class of examples of axial algebras. In this case, the group $\U(A)$ will be a \emph{universal 3\dash transposition group} related to the given Fischer space, as we explain in Theorem~\ref{thm:universal3trans}.
(The well known connection between $3$\dash transposition groups and Fischer spaces is due to Francis Buekenhout \cite{Bue74}.)
This enables us to construct a module for the Matsuo algebra out of every $\U(A)$\dash module; see Theorem \ref{thm:reptomodule}.
This correspondence between modules over Matsuo algebras and $\U(A)$\dash modules is not one-to-one;
it turns out that $1$-eigenvectors in the modules over the Matsuo algebra play a special role (see Corollary~\ref{cor:reptomodule}) and almost always indicate the presence of a
\emph{regular module} as submodule, and this is the content of Theorems~\ref{thm:1eigen1} and~\ref{thm:1eigen2}.

\paragraph*{Acknowledgment.}

We have benefited from fruitful discussions between Jonathan Hall and the second author during a visit at Michigan State University.
We thank the anonymous referee of an earlier version of this paper for his suggestions, which improved the exposition of the results.

\section{Axial algebras} \label{sec:axialalgebras}

This section provides an introduction into the realm of axial algebras. We use the definition by J. Hall, F. Rehren and S. Shpectorov in \cite{HRS15}. This definition of an axial algebra resembles a property satisfied by idempotents in associative and Jordan algebras, namely the existence of a Peirce decomposition. The right multiplication operator of an idempotent in a Jordan algebra is diagonalizable and the decomposition into eigenvectors is compatible with the multiplication, in the sense that the multiplication of eigenvectors is described by a fusion rule.

\begin{defn}
	Let $R$ be a commutative ring with identity. A \emph{fusion rule} is a pair $(\Phi,\star)$ such that $\Phi \subseteq R$ and $\star : \Phi \times \Phi \to 2^\Phi$ is a symmetric map.
\end{defn}

\begin{defn}
	Let $R$ be a commutative ring with identity, $A$ a commutative (not necessarily associative) $R$\dash algebra and $e \in A$ an idempotent. For each $\phi \in R$ we denote the eigenspace of $e$ with eigenvalue $\phi$ by
	\[
		A^e_\phi = \{a \in A \mid ae = \phi a \}.
	\]
	For each subset $\Lambda \subseteq R$, let
	\[
		A^e_\Lambda = \bigoplus_{\phi \in \Lambda} A^e_\phi,
	\]
	with the convention that $A^e_\emptyset = \{0\}$.
	We call $e$ a \emph{$(\Phi,\star)$\dash diagonalizable idempotent} (or an \emph{axis}) for the fusion rule $(\Phi,\star)$ if
	\[
		A = \bigoplus_{\phi \in \Phi} A^e_\phi
	\]
	and
	\[
		A_\phi^e \cdot A_\psi^e \subseteq A_{\phi \star \psi}^e
	\]
	for all $\phi,\psi \in \Phi$. This means that the product of a $\phi$-eigenvector and a $\psi$-eigenvector is a sum of $\chi$-eigenvectors where $\chi$ runs through $\phi \star \psi$.
\end{defn}

\begin{defn}
	A \emph{$(\Phi,\star)$\dash axial algebra} is a pair $(A,\Omega)$ where:
	\begin{enumerate}[(i)]
		\item $A$ is a commutative (not necessarily associative) $R$\dash algebra and,
		\item $\Omega \subset A$ is a generating set of $(\Phi,\star)$\dash axes for $A$.
	\end{enumerate}
	We will often omit the set $\Omega$ in our notation.
\end{defn}

\begin{example}
	Idempotents of Jordan algebras are $\Phi(\alpha)$\dash diagonalizable idempotents for $\alpha = \frac{1}{2}$ where $\Phi(\alpha)$ denotes the \emph{Jordan fusion rule} described by Table \ref{tab:jordanfr} \cite[p.\ 119]{Jac68}.
	\begin{table}
		\renewcommand{\arraystretch}{1.5}
		\[\begin{array}{c|ccc}
			\star & 1 & 0 & \alpha \\ \hline
			1 & \{1\} & \emptyset & \{\alpha\} \\
			0 & \emptyset & \{0\} & \{\alpha\} \\
			\alpha & \{\alpha\} & \{\alpha\} & \{1,0\}
		\end{array}\]
		\caption{the Jordan fusion rule $\Phi(\alpha)$}
		\label{tab:jordanfr}
	\end{table}
\end{example}

The important connection between some axial algebras and groups arises from the special case of a $\mathbb{Z}/2\mathbb{Z}$\dash graded fusion rule.

\begin{defn}
	\begin{enumerate}[(i)]
        \item
            A fusion rule $(\Phi,\star)$ is called \emph{$\mathbb{Z}/2\mathbb{Z}$\dash graded} if $\Phi$ can be partitioned into two subsets $\Phi_+$ and $\Phi_-$ such that
            \begin{align*}
                \phi \star \psi \subseteq \Phi_+ & \text{ whenever } \phi,\psi \in \Phi_+, \\
                \phi \star \psi \subseteq \Phi_+ & \text{ whenever } \phi,\psi \in \Phi_-, \\
                \phi \star \psi \subseteq \Phi_- & \text{ whenever } \phi \in \Phi_+ \text{ and } \psi \in \Phi_-.
            \end{align*}
        \item
            Let $(A,\Omega)$ be a $(\Phi,\star)$-axial algebra for some $\mathbb{Z}/2\mathbb{Z}$-graded fusion rule $(\Phi,\star)$.
            We associate to each $(\Phi,\star)$-axis $e$ of $A$ a \emph{Miyamoto involution} $\tau_e \in \aut(A)$ defined by linearly extending
            \[
            a^{\tau_e} = \begin{cases}
                \phantom{-}a & \text{if $a \in A_{\Phi_+}^e$}, \\
                -a & \text{if $a \in A_{\Phi_-}^e$}.
            \end{cases}
            \]
            Because of the $\mathbb{Z}/2\mathbb{Z}$\dash grading of the fusion rule, these maps define automorphisms of $A$.
            Note that, at this point, we allow $A_{\Phi_-}^e$ to be trivial and hence $\tau_e$ to be trivial.
            However, when $A_{\Phi_-}^e \neq 0$, these automorphisms are indeed involutions.
        \item
            We call the subgroup $\langle \tau_e \mid e \in \Omega \rangle \leq \aut(A)$ the \emph{Miyamoto group} of the axial algebra $(A,\Omega)$,
            and we denote it by $\Miy(A, \Omega)$.
        \item
            We say that $\Omega \subset A$ is \emph{Miyamoto-closed} when it is invariant under $\Miy(A, \Omega)$.
	\end{enumerate}
\end{defn}

\begin{example} \label{ex:jordanengriess}
	\begin{enumerate}[(i)]
		\item The Jordan fusion rule $\Phi(\alpha)$ from Table \ref{tab:jordanfr} is $\mathbb{Z}/2\mathbb{Z}$\dash graded with $\Phi_+=\{1,0\}$ and $\Phi_-=\{\alpha\}$.
		\item \label{ex:jordanengriess:griess} The \emph{Griess algebra} is a 196884\dash dimensional real axial algebra that satisfies the fusion rule from Table \ref{tab:griessfr} \cite[Lemma 8.5.1, p.\ 209]{Iva09}. This fusion rule is $\mathbb{Z}/2\mathbb{Z}$\dash graded with $\Phi_+ = \{1,0,\frac{1}{4}\}$ and $\Phi_-=\{\frac{1}{32}\}$. The Miyamoto involutions of the generating set of axes of the Griess algebra are called \emph{Majorana involutions}. They generate the full automorphism group of the Griess algebra which is better known as the Monster group; see \cite[{Proposition 8.6.2, p.\ 210}]{Iva09} and \cite[p.\ 497]{Tit84}.
		\begin{table}
		\renewcommand{\arraystretch}{1.5}
		\[\begin{array}{c|cccc}
			\star & 1 & 0 & \frac{1}{4} & \frac{1}{32} \\ \hline
			1 & 1 & \emptyset & \frac{1}{4} & \frac{1}{32} \\
			0 & \emptyset & 0 & \frac{1}{4} & \frac{1}{32} \\
			\frac{1}{4} 	& \frac{1}{4} & \frac{1}{4} & 1,0 & \frac{1}{32} \\
			\frac{1}{32} & \frac{1}{32} & \frac{1}{32} & \frac{1}{32} & 1,0,\frac{1}{4}
		\end{array}\]
		\caption{fusion rule of the Griess algebra}
		\label{tab:griessfr}
	\end{table}
	\end{enumerate}
\end{example}

\begin{defn}
	Let $k$ be a field. A \emph{Frobenius axial algebra} is an axial $k$\dash algebra equipped with a bilinear form
	\[
		\langle \cdot , \cdot \rangle : A \times A \to k
	\]
	such that $\langle xa,b \rangle = \langle a , xb \rangle$ for all $a,b,x \in A$. We call this bilinear form the \emph{Frobenius form}.

	It is easy to verify that if $A$ is a Frobenius axial algebra $A$, then for each axis $a \in A$, the eigenspaces
    $A_\phi^a$ and $A_\psi^a$ are perpendicular for distinct $\phi$ and $\psi$ \cite[Proposition~3]{HRS15b}.
\end{defn}

It is already clear from Example \ref{ex:jordanengriess}\ref{ex:jordanengriess:griess} that there exists an important connection between axial algebras and groups. This connection comes from the following situation which is a generalization of a Majorana representation defined in \cite{IPSS10}.

\begin{defn} \label{def:axialrepresentation}
	Let $G$ be a group generated by a set $D$ of elements of order at most~$2$.
    An \emph{axial representation of $(G,D)$} is an isomorphism $\xi \colon G \to \Miy(A,\Omega)$, where
	\begin{enumerate}[(i)]
		\item\label{it:A}
            $(A,\Omega)$ is a $(\Phi,\star)$\dash axial algebra for a $\mathbb{Z}/2\mathbb{Z}$\dash graded fusion rule $(\Phi,\star)$,
		\item
            $\Omega$ is Miyamoto-closed,
		\item\label{it:D}
            $\xi(D) = \{ \tau_e \in \aut(A) \mid e \in \Omega\}$.
	\end{enumerate}
	We will often say that $(A, \Omega)$ \emph{is} an axial representation of $(G,D)$ and
    we will simply identify $G$ with $\Miy(A,\Omega)$ and $D$ with the set $\{ \tau_e \mid e \in \Omega\}$.
\end{defn}

\begin{prop}[{\cite[Lemma 2.4.1]{Reh15}}] \label{prop:tauconj}
	Let $(A,\Omega)$ be an axial representation of $(G,D)$. For each $t \in \aut(A)$ and each $(\Phi,\star)$-axis $a \in A$, $a^t$ is again a $(\Phi,\star)$-axis and $(\tau_a)^t = \tau_{a^t}$. In particular, $(\tau_x)^{\tau_y} = \tau_{x^{\tau_y}}$ for all $x,y \in \Omega$.
\end{prop}

\begin{rem} \label{rem:axialrepresentations}
	\begin{enumerate}[(i)]
		\item
            Of course, every axial algebra $(A,\Omega)$ with a $\mathbb{Z}/2\mathbb{Z}$\dash graded fusion rule such that $\Omega$ is Miyamoto-closed,
            is an axial representation for some $(G,D)$, namely $G = \Miy(A, \Omega)$ and $D = \{ \tau_e \mid e \in \Omega\}$.
            The more interesting question is: which groups $G$ admit an axial representation (for some generating set $D$ of elements of order at most $2$)?
		\item
            In Definition~\ref{def:axialrepresentation}, we explicitly allow the possibility that $D$ contains the trivial element,
            or equivalently, that $\tau_e$ is trivial for some $e \in \Omega$.
		\item
            Let $G$ be a group generated by a set $D$ of elements of order at most~$2$ which is invariant under conjugation.
            If $\xi \colon G \to \Miy(A,\Omega)$ is an isomorphism satisfying~\ref{it:A} and~\ref{it:D} of Definition~\ref{def:axialrepresentation},
            then $\Omega$ is not necessarily Miyamoto-closed, but we can always replace $\Omega$ by the possibly larger set of axes
            \[ \Omega' := \{ e \in A \text{ an axis} \mid \tau_e = \tau_f \text{ for some } f \in \Omega \} \]
            without changing the Miyamoto group.
            It now follows from Proposition~\ref{prop:tauconj} that $\Omega'$ is indeed Miyamoto-closed, and hence
            $(A, \Omega')$ is an axial representation of $(G,D)$.
	\end{enumerate}
\end{rem}

If $(A, \Omega)$ is an axial representation for $(G, D)$, then the map $\tau \colon \Omega \to D \colon e \mapsto \tau_e$ is not necessarily a bijection.
This motivates the following definition, which we have taken from \cite[Definition 6.8]{HSS17}.
\begin{defn}\label{def:uniquetype}
	Let $(A,\Omega)$ be an axial representation of $(G,D)$.
    We call the axial representation \emph{of unique type} if the map $\tau: \Omega \to D: e \mapsto \tau_e$ is a bijection.
\end{defn}

According to \cite[Theorem 6.10]{HSS17}, ``most'' axial representations that satisfy the Jordan fusion rule are of unique type; in particular, those arising from $3$-transposition groups fall in this class (see Proposition~\ref{prop:matsuo-miy}\ref{it:matsuoofuniquetype} below).
Axial representations of unique type behave nicer than others; see, for instance, Proposition~\ref{prop:ofuniquetype} below.

\begin{example}\label{ex:notunique}
	Not every axial representation is of unique type. Indeed, consider the $3$-dimensional algebra spanned by vectors $\mathbbm{1}$, $u$ and $v$ and commutative product defined by $u^2=v^2=\mathbbm{1}$, $uv = 0$ and such that $\mathbbm{1}$ is the identity. This is a Jordan algebra of Clifford type generated by the idempotents $e_1 = \frac{1}{2}+\frac{1}{2}u$, $e_2 = \frac{1}{2}+\frac{1}{2}v$, $\mathbbm{1}-e_1$ and $\mathbbm{1}-e_2$. For the $\mathbb{Z}/2\mathbb{Z}$-grading of the Jordan fusion rule $\Phi(\frac{1}{2})$ we get $\tau_{e_1} = \tau_{\mathbbm{1}-e_1}$ and $\tau_{e_2} = \tau_{\mathbbm{1}-e_2}$. This leads to an axial representation of $(\mathbb{Z}/2\mathbb{Z})^2$, which is not of unique type.

\end{example}

\section{Modules over axial algebras} \label{sec:modulesoveraxialalgebras}

The definition of an axial algebra leads to the following natural definition for modules over axial algebras, our main object of interest.

\begin{defn} \label{def:moduleoveraxialalgebra}
	Let $R$ be a commutative ring with identity and $(A,\Omega)$ a $(\Phi,\star)$\dash axial $R$\dash algebra. Let $M$ be an $R$\dash module equipped with a (right) $R$\dash bilinear action of $A$,
	\[
		M \times A \to M : (m,a) \mapsto m \cdot a.
	\]
	For each $e \in A$ and each $\phi \in R$, let $M_\phi^e = \{ m \in M \mid m \cdot e = \phi m \}$. Define, as usual, for every nonempty $\Lambda \subset R$, $M_\Lambda^e = \bigoplus_{\phi \in \Lambda} M_\phi^e$ and $M_\emptyset^e = \{0\}$. We call $M$ an \emph{$A$\dash module} if, for each $e \in \Omega$,
	\begin{enumerate}[(i)]
		\item there exists a decomposition $M = \bigoplus_{\phi \in \Phi} M_\phi^e$ and
		\item $m \cdot a \in M_{\phi \star \psi}^e$ whenever $m \in M_\phi^e$ and $a \in A_\psi^e$.
	\end{enumerate}
\end{defn}

Observe that $A$ is itself an $A$\dash module; we refer to it as the \emph{regular module} for $A$.

We can extend the definition of Miyamoto involutions and Frobenius forms to arbitrary modules over axial algebras.

\begin{defn}
	Let $A$ be an axial algebra for a $\mathbb{Z}/2\mathbb{Z}$\dash graded fusion rule and $M$ an $A$\dash module. For every $(\Phi,\star)$\dash axis $e \in A$ we define the \emph{Miyamoto involution $\mu_e \in \GL(M)$} of $M$:
	\[
		m^{\mu_e} = \begin{cases}
		\phantom{-}m & \text{if $m \in M_{\Phi_+}^e$}; \\
		-m & \text{if $m \in M_{\Phi_-}^e$}.
		\end{cases}
	\]
	Notice that $\mu_e$ is possibly trivial (namely when $M_{\Phi_-}^e = 0$).
\end{defn}

\begin{defn}
	Let $A$ be an axial $k$\dash algebra for a field $k$. A \emph{Frobenius $A$\dash module} is an $A$\dash module $M$ equipped with a bilinear form, called the \emph{Frobenius form},
	\[
		\langle \cdot , \cdot \rangle : M \times M \to k
	\]
	such that $\langle m \cdot a , n \rangle = \langle m , n \cdot a \rangle$ for all $m,n \in M$ and $a \in A$.
\end{defn}

The following proposition is reminiscent of Maschke's theorem for linear representations of finite groups.

\begin{prop} \label{prop:Maschke}
	Let $A$ be a finite dimensional $k$\dash algebra and $M$ a Frobenius $A$\dash module. Let $N$ be an $A$\dash submodule of $M$, i.e. a $k$\dash subspace for which $N \cdot A \subseteq N$, such that the Frobenius form is non-degenerate on $N$. Then there exists an $A$\dash submodule $N_0$ of $M$ such that $M = N \oplus N_0$.
\end{prop}
\begin{proof}
	Let $N_0 = \{ m \in M \mid \langle m , n \rangle = 0 \text{ for all } n \in N \}$. For all $a \in A$, $n_0 \in N_0$ and $n \in N$, we have $\langle n_0 \cdot a , n \rangle = \langle n_0 , n \cdot a \rangle = 0$ since $n \cdot a \in N$. Hence $n_0 \cdot a \in N_0$ and $N_0$ is an $A$\dash submodule. Since we require the Frobenius form to be non-degenerate on $N$, it follows from the properties of orthogonal complements in finite dimensional vector spaces that $M = N \oplus N_0$.
\end{proof}

In the remainder of this section, we will study modules over axial representations. In Theorem~\ref{thm:moduletorep} below, we will show that every module over an axial representation $(A,\Omega)$ of $(G,D)$ leads to a group representation of a central extension $\U(A,\Omega)$ of $G$. We start by presenting the definition of this central extension $\U(A,\Omega)$ and relating it to $G$.

\begin{defn} \label{def:U}
	Let $(A,\Omega)$ be an axial representation of $(G,D)$. We define the group $\U(A,\Omega)$ as the group with presentation
	\[
		\left\langle t_e \text{ for each } e \in \Omega \Bigm\vert \begin{aligned}
			& (t_e)^2 = 1 && \text{ for all } e \in \Omega \\[-.7ex]
			& (t_x)^{t_y} = t_{x^{\tau_y}} && \text{ for all } x,y \in \Omega
		\end{aligned} \right\rangle.
	\]
	Note that we need $\Omega$ to be Miyamoto-closed in order to define this group.
\end{defn}

We prove some useful properties of this group.

\begin{prop} \label{prop:UA-prop} Let $(A,\Omega)$ be an axial representation of $(G,D)$.
	\begin{enumerate}[\rm(i)]
		\item \label{prop:UA-prop:epi} The map $\tau: \U(A,\Omega) \to G : t_e \mapsto \tau_e$ is a group epimorphism.
		\item \label{prop:UA-prop:central} The group $\U(A,\Omega)$ is a central extension of $G$.
		\item If $\Omega$ is finite, then so is $\U(A,\Omega)$.
	\end{enumerate}
\end{prop}
\begin{proof}
	\begin{enumerate}[(i)]
		\item Consider the map $\tau: \U(A,\Omega) \to G$ defined by $t_e \mapsto \tau_e$ for all $e \in \Omega$. This map is a group homomorphism since all relations that define $\U(A,\Omega)$ hold in $G$ by Proposition \ref{prop:tauconj}. Since the elements $\tau_e$ generate $G$, this map is surjective.
		\item We prove that the kernel of $\tau$ is contained in the center of $\U(A,\Omega)$. Let $u \coloneqq t_{x_1}t_{x_2}\dotsm t_{x_n} \in \ker(\tau)$ where $x_1,x_2,\dots,x_n \in \Omega$. Then $\tau_{x_1}\tau_{x_2}\dotsm\tau_{x_n} = 1$ and thus, for all $e \in \Omega$,
		\[
			t_e^u = t_e^{t_{x_1}t_{x_2}\dotsm t_{x_n}} = t_{e^{\tau_{x_1}\tau_{x_2}\dotsm\tau_{x_n}}} = t_e.
		\]
		Thus $u$ commutes with all $t_e$ and is therefore contained in the center of $\U(A,\Omega)$.
		\item We will show that every element of $\U(A,\Omega)$ can be written as a product of at most $\lvert \Omega \rvert$ elements $t_e$. Since only finitely many such products exist, $\U(A,\Omega)$ must be finite.

		Consider a product of more than $\lvert \Omega \rvert$ elements $t_e$. Some element $t_x$ will then appear at least twice and we can use the relation
		\[
			t_x t_{e_1} t_{e_2} \dotsm t_{e_n} t_x = (t_{e_1} t_{e_2} \dotsm t_{e_n})^{t_x} = t_{(e_1)^{\tau_x}} t_{(e_2)^{\tau_x}} \dotsm t_{(e_n)^{\tau_x}}
		\]
		to rewrite this element as a product of fewer $t_e$'s. \qedhere
	\end{enumerate}
\end{proof}

For axial representations of unique type, we can say even more.

\begin{prop} \label{prop:ofuniquetype} Suppose $(A,\Omega)$ is an axial representation of $(G,D)$ of unique type. Then the following hold.
\begin{enumerate}[\rm(i)]
	\item \label{prop:ofuniquetype:centerless} $G$ is centerless.
	\item $\U(A,\Omega)$ only depends on $G$ and $D$.
	\item \label{prop:ofuniquetype:kernel} The kernel of the group homomorphism $\tau:\U(A,\Omega) \to G : t_e \mapsto \tau_e$ coincides with the center of $\U(A,\Omega)$.
\end{enumerate}
\end{prop}
\begin{proof}
	\begin{enumerate}[(i)]
		\item Suppose $\tau_{x_1} \tau_{x_2} \dotsm \tau_{x_n} \in \Z(G)$. Then, for all $e \in \Omega$,
		\[
			\tau_e = (\tau_e)^{\tau_{x_1} \tau_{x_2} \dotsm \tau_{x_n}} = \tau_{e^{\tau_{x_1} \tau_{x_2} \dotsm \tau_{x_n}}}.
		\]
		Since $(A,\Omega)$ is of unique type, the elements of $\Z(G)$ act trivially on $\Omega$ and hence also on~$A$.
		\item The presentation of Definition \ref{def:U} defining $\U(A,\Omega)$ can be retrieved from $(G,D)$. Since $(A,\Omega)$ is of unique type, the generators can be identified with the set $D$ and the relations are the conjugacy relations between elements of $D$.
		\item This follows from part \ref{prop:ofuniquetype:centerless} and Proposition \ref{prop:UA-prop}\ref{prop:UA-prop:central}. \qedhere
	\end{enumerate}
\end{proof}

\begin{rem}
	It is possible and might seem more natural to define a group $\U(G,D)$ in a similar way as in Definition~\ref{def:U} using the elements of $D$ rather than the idempotents of an axial representation. This group $\U(G,D)$ was already studied by H.~Cuypers and J.~Hall in the context of 3\dash transposition groups \cite[Proposition 2.1]{Hal06}. If $(A,\Omega)$ is an axial representation of $(G,D)$ of unique type, then $\U(G,D) \cong \U(A,\Omega)$. However, for the axial representation of $(\mathbb{Z}/2\mathbb{Z})^2$ from Example \ref{ex:notunique}, which is not of unique type, we have $\U(G,D) \cong (\mathbb{Z}/2\mathbb{Z})^2$ while $\U(A,\Omega) \cong (\mathbb{Z}/2\mathbb{Z})^4$.
\end{rem}

We are now able to state one of our main theorems.

\begin{thm} \label{thm:moduletorep}
	Let $(A,\Omega)$ be an axial representation of $(G,D)$ and let $M$ be an $A$\dash module. The map defined by
	\[
		\mu: \U(A,\Omega) \to \GL(M) : t_e \mapsto \mu_e
	\]
	for all $e \in \Omega$, is a group homomorphism.
\end{thm}

We start by proving the following lemma, the proof of which is inspired by \cite[Lemma~2.4.1]{Reh15}.

\begin{lem} \label{lem:moduletorep}
	For all $m\in M$, $a \in A$ and all $x,y \in \Omega$ we have:
	\begin{enumerate}[\rm(i)]
		\item \label{lem:moduletorep:i} $(m \cdot a)^{\mu_x} = m^{\mu_x} \cdot a^{\tau_x}$,
		\item \label{lem:moduletorep:ii} $(M_\phi^y)^{\mu_x} = M_\phi^{y^{\tau_x}}$,
		\item \label{lem:moduletorep:iii} $(\mu_x)^{\mu_y} = \mu_{x^{\tau_y}}$.
	\end{enumerate}
\end{lem}
\begin{proof}
\begin{enumerate}[(i)]
	\item Let $m \in M^x_{\Phi_+}$ and $a \in A^x_{\Phi_+}$, then $m \cdot a \in M^x_{\Phi_+}$ and therefore $(m \cdot a)^{\mu_x} = m \cdot a$, $m^{\mu_x} = m$ and $a^{\tau_x} = a$. For $m \in M_{\Phi_-}^x$ and $a \in A_{\Phi_+}^x$ (resp. $m \in M_{\Phi_+}^x$ and $a \in A_{\Phi_-}^x$) we have $m \cdot a \in M_{\Phi_-}$ and hence $(m \cdot a)^{\mu_x} = -m \cdot a$, $m^{\mu_x} = - m$ (resp. $m^{\mu_x} = m$) and $a^{\tau_x} = a$ (resp. $a^{\tau_x} = -a$). Finally, let $m \in M_{\Phi_-}^x$ and $a \in A_{\Phi_-}^x$. Now $m \cdot a \in M_{\Phi_+}^x$ and thus $(m \cdot a)^{\mu_x} = m \cdot a$, $m^{\mu_x} = -m$ and $a^{\tau_x} = -a$. Since $M = M^x_{\Phi_+} \oplus M^x_{\Phi_-}$ and $A = A^x_{\Phi_+} \oplus A^x_{\Phi_-}$, the desired property follows by linearity of the action of $A$ on $M$.
	\item Let $m \in M_\phi^y$. By \ref{lem:moduletorep:i}, it follows that $m^{\mu_x} \cdot y^{\tau_x} =(m \cdot y)^{\mu_x} = \phi m^{\mu_x}$, i.e. $m^{\mu_x} \in M_\phi^{y^{\tau_{x}}}$. Therefore $(M_\phi^y)^{\mu_x} \subseteq M_\phi^{y^{\tau_x}}$. Similarly, $(M_\phi^{y^{\tau_x}})^{\mu_x} \subseteq M_\phi^{y^{\tau_x \tau_x}} = M_\phi^y$ and thus $M_\phi^{y^{\tau_x}} = (M_\phi^{y^{\tau_x}})^{\mu_x^2} \subseteq (M_\phi^y)^{\mu_x}$. We find that $(M_\phi^y)^{\mu_x} = M_\phi^{y^{\tau_x}}$.
	\item For $m \in M_\phi^x$, we have $m^{\mu_x} = \varepsilon m$ with $\varepsilon=\pm 1$ only depending on whether $\phi \in \Phi_+$ or $\phi \in \Phi_-$. By \ref{lem:moduletorep:ii}, $m^{\mu_y} \in M_\phi^{x^{\tau_y}}$ and therefore $m^{\mu_y \mu_{x^{\tau_y}}} = \varepsilon m^{\mu_y}$. By applying $\mu_y$,
	\[
		m^{\mu_y \mu_{x^{\tau_y}} \mu_y} = \varepsilon m = m^{\mu_x}
	\]
	follows for all $m \in M_\phi^x$.
	Since $M = \bigoplus M_\phi^x$, by linearity we conclude that $(\mu_x)^{\mu_y} = \mu_{x^{\tau_y}}$. \qedhere
\end{enumerate}
\end{proof}

\begin{proof}[Proof of theorem \ref{thm:moduletorep}]
	By the definition of $\U(A,\Omega)$ and the fact that all $\mu_e$ have order at most 2, the theorem follows from Lemma \ref{lem:moduletorep}\ref{lem:moduletorep:iii}.
\end{proof}

\section{Matsuo algebras} \label{sec:matsuoalgebras}

In Section \ref{sec:modulesovermatsuoalgebras}, we will restrict our attention to the study of modules over Matsuo algebras, a special type of axial algebras.
This section introduces these algebras and explains their connection with Fischer spaces and 3\dash transposition groups.

\begin{defn} \label{def:matsuo}
	\begin{enumerate}[(i)]
		\item
            A \emph{point-line geometry} $\mG$ is a pair $(\mP,\mL)$ where $\mP$ is a set whose elements are called \emph{points} and $\mL$ is a set of subsets of $\mL$. The elements of $\mL$ are called \emph{lines}.
		\item
            Two distinct points $x$ and $y$ of a point-line geometry are said to be \emph{collinear} if there is a line containing both and we denote this by $x \sim y$. We write $x \nsim y$ if $x$ and $y$ are not collinear. If $x \nsim y$ for all $y \neq x$, then we call $x$ an \emph{isolated point}.
		\item
            A \emph{subspace} of point-line geometry $(\mP,\mL)$ is a point-line geometry $(\mP^\prime,\mL^\prime)$ such that:
            \begin{itemize}
                \item $\mP^\prime \subseteq \mP$ and $\mL\subseteq \mL^\prime$,
                \item if $x,y \in \mP^\prime$ and $x,y \in \ell$ for some $\ell \in \mL$, then $\ell \in \mL^\prime$.
            \end{itemize}
            The point-line geometry $(\mP,\mL)$ is a subspace of itself. If $(\mP_1,\mL_1)$ and $(\mP_2,\mL_2)$ are two subspaces of $(\mP,\mL)$, then so is $(\mP_1 \cap \mP_2,\mL_1 \cap \mL_2)$. Therefore, for any set of points and any set of lines, there is a smallest subspace containing those points and lines; we call it the \emph{subspace generated by} those points and lines.
		\item
            Two points $x$ and $y$ of a point-line geometry are called \emph{connected} if there exist points $x=x_0,x_1,\dots,x_n=y$ such that $x_{i-1} \sim x_{i}$ for all $1 \leq i \leq n$. This relation defines an equivalence relation on the set of points of a point-line geometry. The subspaces generated by its equivalence classes are called the \emph{connected components} of the point-line geometry.
		\item
            An \emph{isomorphism} between two point-line geometries $\mG=(\mP,\mL)$ and $\mL=(\mP^\prime,\mL^\prime)$ is a bijection $\theta: \mP \to \mP^\prime$ that induces a bijection between $\mL$ and $\mL^\prime$. We write $\mG \cong \mG^\prime$ and call $\mG$ and $\mG^\prime$ \emph{isomorphic} if an isomorphism between them exists.
		\item
            Let $\mG$ be a point-line geometry such that through any two points there is at most one line and such that each line contains exactly three points. Such a point-line geometry is called a \emph{partial triple system}. If $x$ and $y$ are collinear points in a partial triple system, there is a unique third point on the unique line through $x$ and $y$ and we will denote it by $x \wedge y$.
		\item
            Let $\mG$ be a partial triple system. We call $\mG$ a \emph{Fischer space} if each subspace generated by two distinct intersecting lines is isomorphic to the dual affine plane of order 2 or the affine plane of order 3. A sketch of these two point-line geometries is given in Figure \ref{fig:fischerspaces}.
	\end{enumerate}
\end{defn}

\begin{figure}[ht]
\centering
\begin{tikzpicture}[yshift=5cm,scale=0.8]
	\draw[thick, name path=l1] (0,0) -- (2,6) -- (4,0);
	\draw[thick, name path=l2] (-0.5,0.5) -- (5.5,2.5) -- (-0.5,4.5);
	\path[name intersections={of=l1 and l2, by={2,3,4,5}}];
	\coordinate (1) at (2,6);
	\coordinate (6) at (5.5,2.5);
	\node[fill = white, circle, inner sep = 2.5pt] at (1){};
	\draw[thick] (1) circle (1.5pt);
	\node[above, outer sep = 1pt] at (1){};
	\node[fill = white, circle, inner sep = 2.5pt] at (2){};
	\draw[thick] (2) circle (1.5pt);
	\node[above left, outer sep = 1pt] at (2){};
	\node[fill = white, circle, inner sep = 2.5pt] at (3){};
	\draw[thick] (3) circle (1.5pt);
	\node[above left, outer sep = 1pt] at (3){};
	\node[fill = white, circle, inner sep = 2.5pt] at (4){};
	\draw[thick] (4) circle (1.5pt);
	\node[below right, outer sep = 1pt] at (4){};
	\node[fill = white, circle, inner sep = 2.5pt] at (5){};
	\draw[thick] (5) circle (1.5pt);
	\node[above right, outer sep = 1pt] at (5){};
	\node[fill = white, circle, inner sep = 2.5pt] at (6){};
	\draw[thick] (6) circle (1.5pt);
	\node[right, outer sep = 1pt] at (6){};
\begin{scope}[scale=0.70,shift={(11,1.5)}]
	\clip (-2,-2) rectangle (8,8);
	\node[fill = white, inner sep = 2.5pt, circle] (1) at (0,6){};
	\node[fill = white, inner sep = 2.5pt, circle] (2) at (3,6){};
	\node[fill = white, inner sep = 2.5pt, circle] (3) at (6,6){};
	\node[fill = white, inner sep = 2.5pt, circle] (4) at (0,3){};
	\node[fill = white, inner sep = 2.5pt, circle] (5) at (3,3){};
	\node[fill = white, inner sep = 2.5pt, circle] (6) at (6,3){};
	\node[fill = white, inner sep = 2.5pt, circle] (7) at (0,0){};
	\node[fill = white, inner sep = 2.5pt, circle] (8) at (3,0){};
	\node[fill = white, inner sep = 2.5pt, circle] (9) at (6,0){};
	\draw[thick] (1) circle (2pt);
	\draw[thick] (2) circle (2pt);
	\draw[thick] (3) circle (2pt);
	\draw[thick] (4) circle (2pt);
	\draw[thick] (5) circle (2pt);
	\draw[thick] (6) circle (2pt);
	\draw[thick] (7) circle (2pt);
	\draw[thick] (8) circle (2pt);
	\draw[thick] (9) circle (2pt);
	\draw[thick] (1) -- (2) -- (3);
	\draw[thick] (4) -- (5) -- (6);
	\draw[thick] (7) -- (8) -- (9);
	\draw[thick] (1) -- (4) -- (7);
	\draw[thick] (2) -- (5) -- (8);
	\draw[thick] (3) -- (6) -- (9);
	\draw[thick] (1) -- (5) -- (9);
	\draw[thick] (3) -- (5) -- (7);
	\draw[thick] (1) -- (8);
	\draw[thick] (8) to [out = -63.435,in = -26.565, looseness=2.5] (6);
	\draw[thick] (6) -- (1);
	\draw[thick] (7) -- (2);
	\draw[thick] (2) to [out = 63.435,in = 26.565, looseness=2.5] (6);
	\draw[thick] (6) -- (7);
	\draw[thick] (9) -- (2);
	\draw[thick] (2) to [out = {180-63.435},in = {180-26.565}, looseness=2.5] (4);
	\draw[thick] (4) -- (9);
	\draw[thick] (3) -- (8);
	\draw[thick] (8) to [out = {180+63.435},in = {180+26.565}, looseness=2.5] (4);
	\draw[thick] (4) -- (3);
	\node[above left] at (1){};
	\node[outer sep = 6pt, above] at (2){};
	\node[above right] at (3){};
	\node[outer sep = 5pt, left] at (4){};
	\node[outer sep = 5pt, anchor = {180+20}] at (5){};
	\node[outer sep = 5pt, right] at (6){};
	\node[below left] at (7){};
	\node[outer sep = 6pt, below] at (8){};
	\node[below right] at (9){};
\end{scope}
\end{tikzpicture}
\caption{the dual affine plane of order 2 and the affine plane of order 3}\label{fig:fischerspaces}
\end{figure}
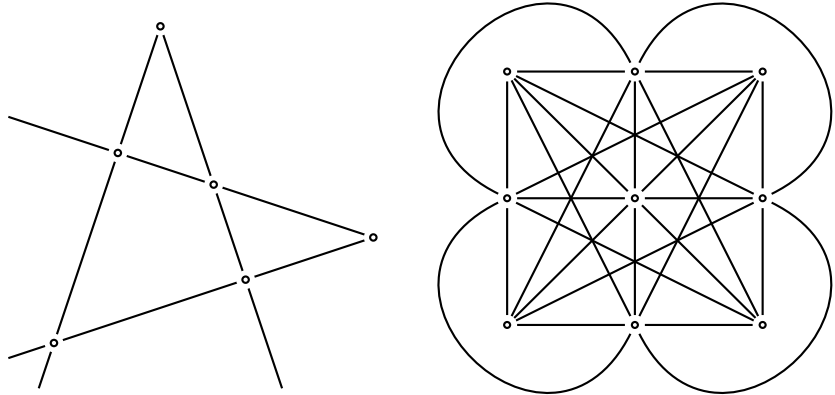

The main motivation for studying Fischer spaces is their connection with $3$-transposition groups, due to F. Buekenhout~\cite{Bue74};
see Proposition~\ref{prop:Bue} below.

\begin{defn}[\cite{Asc97}]
	A \emph{3\dash transposition group} is a pair $(G,D)$ where $D$ is a generating set of involutions of $G$ closed under conjugation such that the order of the product of any two elements of $D$ is at most $3$.
\end{defn}

\begin{example}
	\begin{enumerate}[(i)]
		\item All symmetric groups together with their set of transpositions form a 3\dash transposition group.
		\item The Fischer groups $Fi_{22}$, $Fi_{23}$ and $Fi_{24}$ are 3\dash transposition groups.
	\end{enumerate}
\end{example}


In the connection between Fischer spaces and $3$-transposition groups that we will need, it will be of importance to treat isolated points in Fischer spaces with some care.

\begin{defn}\label{def:3t-Fs}
	\begin{enumerate}[(i)]
		\item
            Let $(\mP,\mL)$ be a Fischer space and let $\mP^\prime \subseteq \mP$ be the set of isolated points of $(\mP,\mL)$.
            Then $(\mP \setminus \mP^\prime,\mL)$ is clearly a Fischer space without isolated points
            that we will denote by $(\mP,\mL)^\circ$.
		\item\label{it:tau-Fs}
            If $\mG=(\mP,\mL)$ is a Fischer space, then we associate with each point $x \in \mP$ an automorphism $\tau_x \in \aut(\mG)$
            defined as
            \[ \tau_x \colon \mP \to \mP \colon y \mapsto \begin{cases}
                x \wedge y & \text{if } y \sim x, \\
                y & \text{if } y \nsim x .
            \end{cases} \]
            Notice that $\tau_x$ is an involution unless $x$ is an isolated point (in which case $\tau_x$ is trivial).
            These involutions not only leave the set of points of the Fischer space invariant but also map collinear points to collinear points;
            hence they induce automorphisms of the Fischer space.

            Now let $D = \{ \tau_x \mid x \in \mP \text{ and } \tau_x \neq 1 \}$ and $G = \langle D \rangle \leq \aut(\mG)$;
            then we define $f(\mG) := (G, D)$.
		\item
            Let $(G,D)$ be a 3\dash transposition group. Then we write $(G,D)^\circ$ for the 3\dash transposition group $(G/\Z(G), \{ d\Z(G) \mid d \in D\setminus \Z(G)\})$.
		\item
            Let $(G,D)$ be a 3\dash transposition group. Let $\mP = D$ and let $\mL=\{\{c,d,c^d = d^c\} \mid \order(cd) = 3 \}$. Then $g(G,D) \coloneqq (\mP,\mL)$ is a point-line geometry.
	\end{enumerate}
\end{defn}

\begin{prop}[\cite{Bue74}] \label{prop:Bue}
	Let $\mG$ be a Fischer space and let $(G,D)$ be a 3\dash transposition group. Then
	\begin{enumerate}[\rm(i)]
		\item
            $g(G,D)$ is a Fischer space,
		\item
            $f(\mG)$ is a 3\dash transposition group,
		\item\label{it:fg}
            $f(g(G,D)) \cong (G,D)^\circ$,
		\item
            $g(f(\mG)) \cong \mG^\circ$.
	\end{enumerate}
\end{prop}

\begin{example} \label{ex:3trans}
	\begin{enumerate}[(i)]
		\item The Fischer space corresponding to the 3\dash transposition group $S_4$ is the dual affine plane of order 2.
		\item \label{ex:3trans:AG23} The 3\dash transposition group related to the affine plane of order 3 is a semidirect product $3^2:2$ where the action is given by inversion.
	\end{enumerate}
\end{example}

We now present the definition of Matsuo algebras, which will be instances of axial algebras; see Proposition~\ref{prop:matsuo-eigenspaces} below.
When the Matsuo algebra arises from a Fischer space, its fusion rules will be $\mathbb{Z}/2\mathbb{Z}$-graded; see Proposition~\ref{prop:matsuofischer} below.

\begin{defn}
    Let $k$ be a field with $\Char(k) \neq 2$, let $\alpha \in k \setminus \{0,1\}$ and let $\mG = (\mP, \mL)$ be a partial triple system.
    Define the Matsuo algebra $M_\alpha(\mG)$ as the $k$\dash vector space with basis $\mP$ and multiplication defined by linearly extending
    \[
        xy = \begin{cases}
            x & \text{if } x=y, \\
            0 & \text{if } x \nsim y, \\
            \frac{\alpha}{2} (x + y - x \wedge y) & \text{if } x \sim y,
        \end{cases}
    \]
    for all $x,y \in \mP$.
\end{defn}

The following proposition gives us a decomposition of a Matsuo algebra as a direct sum of eigenspaces for any $x \in \mP$.

\begin{prop}[{\cite[Theorem 6.2]{HRS15b}}] \label{prop:matsuo-eigenspaces} For each $x \in \mP$, the eigenspaces of $x$ in $M_\alpha(\mG)$ are
\begin{align*}
	&\langle x \rangle \text{, its 1-eigenspace,} \\
	&\langle y + x \wedge y - \alpha x \mid y \sim s \rangle \oplus \langle y \mid y \nsim x \rangle \text{, its 0-eigenspace,} \\
	&\langle y - x \wedge y \mid y \sim x \rangle \text{, its $\alpha$-eigenspace,}
\end{align*}
and the algebra $M_\alpha(\mG)$ decomposes as a direct sum of these eigenspaces.
\end{prop}

These decompositions satisfy the Jordan fusion rule $\Phi(\alpha)$ precisely when $\mG$ is a Fischer space:
\begin{prop}[{\cite[Theorem 6.5]{HRS15b}}] \label{prop:matsuofischer}
    Let $\mG = (\mP, \mL)$ be a partial triple system, with corresponding Matsuo algebra $A = M_\alpha(\mG)$.
    Then the Matsuo algebra $(A ,\mP)$ is a $\Phi(\alpha)$\dash axial algebra if and only if $\mG$ is a Fischer space.
\end{prop}

Since the Jordan fusion rule $\Phi(\alpha)$ is $\mathbb{Z}/2\mathbb{Z}$\dash graded with $\Phi(\alpha)_+ = \{1,0\}$ and $\Phi(\alpha)_- = \{\alpha\}$,
we can consider the Miyamoto involutions $\tau_x \in \aut(M_\alpha(\mG))$ for each $x \in \mP$.
Isolated points of the Fischer space should be treated with some care;
they do not pose any serious difficulties, however, since they would give rise to trivial Miyamoto involutions.
For simplicity, we nevertheless exclude this situation.
\begin{prop}[{\cite[Theorem 6.4]{HRS15b}}] \label{prop:matsuo-miy}
    Let $\mG = (\mP, \mL)$ be a Fischer space without isolated points, with corresponding Matsuo algebra $A = M_\alpha(\mG)$.
    Then:
	\begin{enumerate}[\rm(i)]
        \item\label{it:miy1}
            The pair $(A, \mP)$ is an axial representation of the 3\dash transposition group $f(\mG)$.
        \item\label{it:miy2}
            For each $x \in \mP$, the Miyamoto involution $\tau_x \in \aut(A)$ acts on $\mP$ as
            the automorphism $\tau_x$ introduced in Definition~\textup{\ref{def:3t-Fs}\ref{it:tau-Fs}}.
        \item\label{it:matsuoofuniquetype}
            The axial representation $(A, \mP)$ is of unique type.
    \end{enumerate}
\end{prop}
\begin{proof}
    Statements~\ref{it:miy1} and~\ref{it:miy2} follow from~\cite[Theorem 6.4]{HRS15b}.
    We now show~\ref{it:matsuoofuniquetype}.
    So let $D = \{ \tau_x \mid x \in \mP\}\subseteq \aut(A)$;
	we have to show that the map $\tau : \mP \to D : x \mapsto \tau_x$ is injective.
    Suppose $\tau_x = \tau_y$ for some $x,y \in \mP$; then by~\ref{it:miy2}, $\tau_x$ and $\tau_y$ induce the same automorphism of the Fischer space $\mG$.
    Since $\mG$ has no isolated points, there exists some point $z \in \mP$ such that $z \sim x$.
    Then $z^{\tau_y} = z^{\tau_x} = x \wedge z$ and therefore both $x$ and $y$ are the third point on the line through $z$ and $x \wedge z$.
\end{proof}

\begin{rem}
    Let $\mG = (\mP, \mL)$ be a Fischer space without isolated points, with corresponding Matsuo algebra $A = M_\alpha(\mG)$.
    By Proposition~\ref{prop:matsuo-miy}\ref{it:miy2}, we can view the Miyamoto group $\Miy(A, \mP) = \langle \tau_x \mid x \in \mP \rangle$
    as a subgroup of either $\aut(A)$ or $\aut(\mG)$, whichever is the most convenient.
\end{rem}

We now start from a $3$-transposition group $(G,D)$ and we would like to construct an axial representation for $(G,D)$.
Again, the situation giving rise to isolated points of the corresponding Fischer space should be treated with some care,
and this works out nicely when we assume that $G$ is centerless.

\begin{prop} \label{prop:Bue2}
	Let $(G,D)$ be a 3\dash transposition group with $Z(G)=1$.
    Then $(G,D)$ has an axial representation $(M_\alpha(\mG),\mP)$ for the Fischer space $\mG = g(G, D)$.
\end{prop}
\begin{proof}
	Since $G$ is centerless, $(G,D) = (G,D)^\circ$. The statement now follows from Proposition~\ref{prop:Bue}\ref{it:fg}
    together with Proposition~\ref{prop:matsuo-miy}\ref{it:miy1}.
\end{proof}

The following theorem, which is our main result in this section, shows that the group $\U(M_\alpha(\mG),\mP)$
can be interpreted as a universal 3\dash transposition group for the given Fischer space $\mG$.
Recall the definition of $\U(A)$ from Definition~\ref{def:U} above.

\begin{thm} \label{thm:universal3trans}
	Let $\mG=(\mP,\mL)$ be a Fischer space without isolated points.
    Consider the axial representation $\mA = (M_\alpha(\mG),\mP)$ of the 3\dash transposition group $(G,D) \coloneqq f(\mG)$.
    Then:
	\begin{enumerate}[\rm(i)]
		\item $\bigl( \U(\mA),\{ t_x \mid x \in \mP \} \bigr)$ is a 3\dash transposition group and $g\bigl( \U(\mA),\{ t_x \mid x \in \mP \} \bigr) \cong \mG$.
		\item Let $(G^\prime,D^\prime)$ be a 3\dash transposition group such that $g(G^\prime,D^\prime) \cong \mG$. Let $\varphi : \mP \to D^\prime$ be an isomorphism between $\mG$ and $g(G^\prime,D^\prime)$. The map defined by
		\[
			\theta : \U(\mA) \to G^\prime : t_x \mapsto \varphi(x)
		\]
		is a group epimorphism and $G'/\Z(G') \cong \U(\mA)/\Z(\U(\mA)) \cong G$.
	\end{enumerate}
\end{thm}
\begin{proof}
	\begin{enumerate}[(i)]
		\item
            Since all relations that hold in $\U(\mA)$ have even length and all $t_x$ have order at most $2$, all $t_x$ have order exactly $2$.
            Therefore $\{ t_x \mid x \in \mP \}$ is, by definition of $\U(\mA)$, a generating set of involutions invariant under conjugation.
            Since $\mG$ has no isolated points, the axial representation $\mA$ is of unique type and all $\tau_x$ are different.
            By Proposition \ref{prop:UA-prop}\ref{prop:UA-prop:epi}, also all $t_x$ must be different.

            Let $x,y \in \mP$; then, by Proposition \ref{prop:matsuo-miy}, either $x=y$ or $x^{\tau_y} = x$ if $x \nsim y$ or $x^{\tau_y} = x \wedge y$ if $x \sim y$.
            In the first case, $t_xt_y =1$. In the second case, $(t_x t_y)^2 = t_x t_{x^{\tau_y}} = 1$.
            In the third case, $y^{\tau_x} = x \wedge y$ and thus $(t_xt_y)^3=t_{x^{\tau_y}} t_{y^{\tau_x}} = (t_{x \wedge y})^2 = 1$.
            This proves that $(\U(\mA),\{ t_x \mid x \in \mP \})$ is indeed a 3\dash transposition group.

            Let $x,y,z \in \mP$; then $t_x$, $t_y$ and $t_z$ lie on a line in $g(\U(\mA),\{ t_x \mid x \in \mP \})$ if and only if $\order(t_x t_y) = 3$ and $(t_x)^{t_y} = t_z$.
            By our previous arguments, this happens if and only if $x \sim y$ and $z = x \wedge y$.
		\item
            For each $x \in \mP$, $\varphi(x) \in D^\prime$ is an involution.
            Let $x,y \in \mP$.
            If $x \sim y$, then, by definition of $g(G^\prime,D^\prime)$, $\varphi(x)^{\varphi(y)} = \varphi(x\wedge y) = \varphi(x^{\tau_y})$.
            If $x \nsim y$, then $\varphi(x)\varphi(y)$ has order 2 and $\varphi(x)^{\varphi(y)} = \varphi(x) = \varphi(x^{\tau_y})$.
            This proves that $\theta$ is a group epimorphism.

            Consider the map $\theta^\prime : G^\prime \to G$ defined by $\varphi(x) \mapsto \tau_x$.
            Then $\ker(\theta^\prime \circ \theta) = \Z(\U(\mA))$ by Proposition \ref{prop:ofuniquetype}\ref{prop:ofuniquetype:kernel}.
            Because $\theta$ is surjective, $\ker(\theta^\prime) \leq \Z(G^\prime)$ and, since $G$ is centerless, $\ker(\theta^\prime) = \Z(G^\prime)$.
            The isomorphism $G \cong \U(\mA)/\Z(\U(\mA))$ follows from Proposition \ref{prop:ofuniquetype}\ref{prop:ofuniquetype:kernel}
            or by applying the previous argument to the $3$\dash transposition group $(G^\prime,D^\prime) = \bigl( \U(\mA),\{t_x \mid x \in \mP\} \bigr)$.
        \qedhere
	\end{enumerate}
\end{proof}

\begin{example}
	Let $\mathcal{G}=(\mathcal{P},\mathcal{L})$ be the affine plane of order 3. Then $\U(M_\alpha(\mathcal{G}),\mathcal{P})$ is a semidirect product $3^2:S_3$ which is a central extension of the group $3^2:2$ from Example~\ref{ex:3trans}\ref{ex:3trans:AG23} by a cyclic group of order 3.
\end{example}

\section{Modules over Matsuo algebras} \label{sec:modulesovermatsuoalgebras}

Theorem \ref{thm:moduletorep} tells us that we can get a group representation out of a module over an axial algebra. For Matsuo algebras, the converse will also be true. In Theorem \ref{thm:reptomodule} we construct a module over a Matsuo algebra from such a group representation.

\begin{thm} \label{thm:reptomodule}
	Let $\mG=(\mP,\mL)$ be a Fischer space. Let $k$ be a field with $\Char(k) \neq 2$ and $\alpha \in k\setminus \{0,1\}$. Consider the Matsuo algebra $M_\alpha(\mG)$. Let $V$ be a $k$\dash vector space and $\rho: \U(M_\alpha(\mG),\mP) \to \GL(V)$ a group homomorphism. Since $\rho(t_x)^2 = 1$ for every $x \in \mP$ and $\Char(k) \neq 2$, we can decompose $V$ as a direct sum of the $1$\dash\ and $(-1)$\dash eigenspace of $\rho(t_x)$. The action defined by linearly extending
	\[
		v \cdot x = \begin{cases}
 			0 & \text{if } v^{\rho(t_x)} = v; \\
 			\alpha v & \text{if } v^{\rho(t_x)} = -v,
 		\end{cases}
	\]
	for each $x \in \mP$, equips $V$ with the structure of an $M_\alpha(\mG)$\dash module.
\end{thm}

We start by mimicking Lemma \ref{lem:moduletorep}(i) with $\mu_x$ replaced by $\rho(t_x)$.

\begin{lem} \label{lem:reptomodule}
	For every $v \in V$, $a \in M_\alpha(\mG)$ and $x \in \mP$, $(v \cdot a)^{\rho(t_x)} = v^{\rho(t_x)} \cdot a^{\tau_x}$.
\end{lem}
\begin{proof}
	Since $M_\alpha(\mG)$ is spanned by the elements of $\mP$ and the action of $M_\alpha(\mG)$ is linear by definition, we may assume that $a = y \in \mP$. Since $V$ is decomposable into a $1$\dash\ and $(-1)$\dash eigenspace of $\rho(t_x)$, it suffices to consider the cases where $v$ is a $1$\dash\ or $(-1)$\dash eigenvector of $\rho(t_x)$. If $v^{\rho(t_y)} = v$ then $v \cdot y = 0$ and $(v \cdot y)^{\rho(t_x)} = 0$. Moreover
	\[
		\left(v^{\rho(t_x)}\right)^{\rho(t_{y^{\tau_x}})} = v^{\rho(t_x)\rho(t_x t_y t_x)} = v^{\rho(t_x)}
	\]
	and hence $v^{\rho(t_x)} \cdot y^{\tau_x} = 0$. In the second case, $v^{\rho(t_y)} = -v$ and therefore $v \cdot y = \alpha v$ and $(v \cdot y)^{\rho(t_x)} = \alpha v^{\rho(t_x)}$. Now
	\[
		\left( v^{\rho(t_x)} \right)^{\rho(t_{y^{\tau_x}})} = v^{\rho(t_x)\rho(t_x t_y t_x)} = -v^{\rho(t_x)}
	\]
	and thus $v^{\rho(t_x)} \cdot y^{\tau_x} = \alpha v^{\rho(t_x)}$. In both cases, $(v \cdot y)^{\rho(t_x)} = v^{\rho(t_x)} \cdot y^{\tau_x}$.
\end{proof}

\begin{proof}[Proof of Theorem \ref{thm:reptomodule}]
	Let $x \in \mP$. The vector space $V$ decomposes into a $1$\dash and $(-1)$\dash eigenspace of $\rho(t_x)$. By definition of the action of $x$, $V$ decomposes as $V = V_0^x \oplus V_\alpha^x$ where $V_0^x$ (resp. $V_\alpha^x$) is the $1$\dash eigenspace (resp. $(-1)$\dash eigenspace) of $\rho(t_x)$. It only remains to verify that the fusion rule is satisfied. Let $A = M_\alpha(\mG)$. For $v \in V_0^x$ and $a \in A^x_{\{1,0\}}$ (resp.\ $v \in V_\alpha^x$ and $a \in A^x_\alpha$) we have $v^{\rho(t_x)} = v$ (resp.\ $v^{\rho(t_x)} = -v$) and $a^{\tau_x} = a$ (resp.\ $a^{\tau_x} = -a$). By Lemma \ref{lem:reptomodule},
	\[
		(v \cdot a)^{\rho(t_x)} = v^{\rho(t_x)} \cdot a^{\tau_x} = v \cdot a
	\]
	and thus $v \cdot a$ is a $1$\dash eigenvector of $\rho(t_x)$. Therefore $v \cdot a$ belongs to $V_0^x$. If $v \in V_0^x$ and $a \in A^x_{\alpha}$ (resp. $v \in V_\alpha^x$ and $a \in A^x_{\{1,0\}}$), then $v^{\rho(t_x)} = v$ (resp. $v^{\rho(t_x)} = -v$) and $a^{\tau_x} = -a$ (resp. $a^{\tau_x}=a$). In both cases, by Lemma \ref{lem:reptomodule},
	\[
		(v \cdot a)^{\rho(t_x)} = v^{\rho(t_x)} \cdot a^{\tau_x} = -v \cdot a.
	\]
	Therefore $v \cdot a$ is a $1$\dash eigenvector of $\rho(t_x)$ and hence $v \cdot a \in V_\alpha^x$.
\end{proof}

\begin{cor} \label{cor:reptomodule}
	Let $\mG=(\mP,\mL)$ be a Fischer space and let $A=(M_\alpha(\mG),\mP)$ be its Matsuo algebra. There is a bijective correspondence between (isomorphism classes of) modules $M$ for $A$ such that $M_1^x = \{0\}$ for all $x \in \mP$ and (isomorphism classes of) linear representations of $\U(A)$.
\end{cor}
\begin{proof}
    Using Theorem \ref{thm:reptomodule}, we can associate to each linear representation $\rho$ of $\U(A)$ an $A$-module $M_\rho$ with $(M_\rho)^x_1 = \{0\}$ for all $x$.
    We prove that the resulting map $\zeta \colon \rho \mapsto M_\rho$ is a bijection from the set of isomorphism classes of linear representations of $\U(A)$ to the set
    of isomorphism classes of $A$-modules with trivial $1$-eigenspaces.

    To show that $\zeta$ is onto, let $M$ be an arbitrary $A$-module with $M_1^x = \{0\}$ for all $x \in \mP$.
    By Theorem \ref{thm:moduletorep}, there is an associated linear representation $\rho$ of the group $\U(A)$;
    we claim that $M_\rho \cong M$.
    Indeed, since $\rho(t_x) = \mu_x$ and $M_1^x = \{0\}$,
    we see that $x \cdot m = 0$ (resp. $x \cdot m = \alpha m$) if and only if $m^{\rho(t_x)} = m$ (resp. $m^{\rho(t_x)} = -m$) for all $m \in M$ and $x \in \mP$. By definition of $M_\rho$, this proves the claim.

    On the other hand, $\rho$ is uniquely determined by its action on the eigenspaces of $M_\rho$, and hence $\zeta$ is also injective.
\end{proof}

More generally, given any module $M$ over the Matsuo algebra $M_\alpha(\mG)$, Theorem~\ref{thm:moduletorep} gives us a representation $\rho$ of the group $\U(M_\alpha(\mG),\mP)$. The module constructed by Theorem \ref{thm:reptomodule} out of $\rho$ resembles $M$, with the important difference that all $1$\dash eigenvectors of an axis $x \in \mP$ have become $0$\dash eigenvectors. (The action on the $0$\dash\ and $\alpha$\dash eigenvectors remains unchanged.) This leaves of course the question what the role of a $1$\dash eigenvector is inside a module. Since a regular module, i.e., the axial algebra as a module over itself, contains non-trivial 1\dash eigenspaces, we definitely want to allow the existence of $1$\dash eigenspaces in Definition \ref{def:moduleoveraxialalgebra}. We will now prove that, under certain conditions, this is the only way a $1$\dash eigenspace can turn up in a module over a Matsuo algebra.

From now on, we would like to restrict to Matsuo algebras over connected Fischer spaces without isolated points.
This is not a serious restriction, as the following proposition allows to generalize results to arbitrary Fischer spaces without isolated points
by looking at their connected components.

\begin{prop}[{\cite[Theorem 6.2]{HRS15b}}]
	Let $\mG$ be a Fischer space and let $\{\mG_i \mid i \in I\}$ be the set of its connected components. Then $M_\alpha(\mG_i) \cong \bigoplus_{i \in I} M_\alpha(\mG_i)$.
\end{prop}

The connectedness of a Fischer space has the following implications on its Matsuo algebra.

\begin{lem} \label{lem:trans}
	Let $\mG=(\mP,\mL)$ be a Fischer space and let $M_\alpha(\mG)$ be its Matsuo algebra. Let $D = \{\tau_x \mid x \in \mP\}$ and $G = \langle D \rangle \leq \aut(M_\alpha(\mG))$. The following statements are equivalent.
	\begin{enumerate}[\rm(a)]
		\item The Fischer space $\mG$ is connected.
		\item The action of $G$ on $\mP$ is transitive.
		\item The set $D$ of Miyamoto involutions is a conjugacy class of $G$.
		\item The set $\{t_x \mid x \in \mP\}$ is a conjugacy class of $\U(M_\alpha(\mG),\mP)$.
	\end{enumerate}
\end{lem}

\begin{proof}
	Suppose the Fischer space $\mG$ is connected. Let $x,y \in \mP$. Then there exists a path $x,x_1,\dots,x_n,y$ from $x$ to $y$. Now
	\[
		x^{\tau_{x \wedge x_1} \tau_{x_1 \wedge x_2} \dotsm \tau_{x_n \wedge y}} = y
	\]
	and therefore $G$ acts transitively on $\mP$. Conversely, by Proposition \ref{prop:matsuo-miy}, all Miyamoto involutions stabilize the connected components of the Fischer space. Therefore, the Fischer space is connected when $G$ acts transitively on $\mP$. The equivalence between (b) and (c) follows from Proposition \ref{prop:tauconj}. By definition of the group $\U(M_\alpha(\mG),\mP)$, it follows that (c) and (d) are equivalent.
\end{proof}

From now on, let $\mG=(\mP,\mL)$ be a connected Fischer space without isolated points. Consider the axial representation $(M_\alpha(\mG),\mP)$ of the 3\dash transposition group $(G,D) \coloneqq f(\mG)$. Let $M$ be an $M_\alpha(\mG)$\dash module and suppose $M_1^x \neq \{0\}$ for some (and hence every) axis $x \in \mP$. We will prove in Theorem \ref{thm:1eigen1} that, under certain conditions, $M$ contains a submodule that is a quotient of $M_\alpha(\mG)$ as an $M_\alpha(\mG)$\dash module. First, we will construct the submodule. We introduce some notation.

\begin{defn} \label{def:1eigenspace}
    Let $M$ be an $M_\alpha(\mG)$\dash module and suppose $M_1^x \neq \{0\}$ for each $x \in \mP$.
	\begin{enumerate}[(i)]
		\item Let
			\[U \coloneqq \langle \mu_x \mid x \in \mP\rangle \leq \GL(M).\]
		Denote the group $\U(M_\alpha(\mG),\mP)$ by $T$ and consider, as in Proposition \ref{prop:UA-prop}\ref{prop:UA-prop:epi} and Theorem \ref{thm:moduletorep}, the group epimorphisms defined by
		\begin{align*}
			&\mu : T \to U : t_x \mapsto \mu_x, \\
			&\tau : T \to G : t_x \mapsto \tau_x.
		\end{align*}
		\item For each $x \in \mP$, we define a subgroup $\U_x = \mu(C_{T}(t_x))$ of $U$, where $C_{T}(t_x)$ is the centraliser of $t_x$ in $T$.
		\item Fix $x \in \mP$. Let $m \in M_1^x$ with $m \neq 0$. Let
		\[
			m_x = \sum_{\mu \in U_x} m^\mu.
		\]
		Because we assume that $\mG$ is connected, the group $T$ acts transitively on the set $\{t_x \mid x \in \mP\}$. Thus for every $y \in \mP$ there exists $t\in T$ such that $(t_x)^t = t_y$ and we define
		\begin{equation} \label{eq:defmy}
			m_y = (m_x)^{\mu(t)}.\tag{$\star$}
		\end{equation}
		By Lemma \ref{lem:1eigen1}\ref{lem:1eigen1:iv} below, this definition of $m_y$ is independent of the choice of $t$.
	\end{enumerate}
\end{defn}

\begin{lem} \label{lem:1eigen1} Let $y \in \mP$, $t \in T$ and $\phi \in \{1,0,\alpha\}$. Suppose $t^\prime \in T$ such that $(t_x)^t = (t_x)^{t^\prime}$. Then the following hold:
	\begin{enumerate}[\rm(i)]
		\item $(t_y)^t = t_{y^{\tau(t)}}$,
		\item \label{lem:1eigen1:ii} $(M_\phi^y)^{\mu(t)} = M_\phi^{y^{\tau(t)}}$,
		\item \label{lem:1eigen1:iii} $U_x\mu(t) = U_x\mu(t^\prime)$,
		\item \label{lem:1eigen1:iv} $(m_x)^{\mu(t)} = (m_x)^{\mu(t^\prime)}$,
		\item \label{lem:1eigen1:v} $m_y \in M_1^y$,
		\item \label{lem:1eigen1:vi} $(m_y)^{\mu(t)} = m_{y^{\tau(t)}}$.
	\end{enumerate}
\end{lem}

\begin{proof}
	\begin{enumerate}[(i)]
		\item This follows by definition of $\U(M_\alpha(\mG),\mP)$.
		\item By Lemma \ref{lem:moduletorep}\ref{lem:moduletorep:ii}, this follows because $\mu$ and $\tau$ are group homomorphisms.
		\item Since $C_{\U(M_\alpha(\mG),\mP)}(t_x)t = C_{\U(M_\alpha(\mG),\mP)}(t_x)t^\prime$ whenever $(t_x)^t=(t_x)^{t^\prime}$, also $U_x\mu(t) = U_x\mu(t^\prime)$.
		\item By \ref{lem:1eigen1:iii},
		\[
			(m_x)^{\mu(t)} = \sum_{\mu \in U_x\mu(t)} m^\mu = \sum_{\mu \in U_x\mu(t^\prime)} m^\mu = (m_x)^{\mu(t^\prime)}.
		\]
		\item Let $t \in C_{\U(M_\alpha(\mG),\mP)}(t_x)$. Since $t_x = (t_x)^t = t_{x^{\tau(t)}}$, it follows from Proposition \ref{prop:matsuo-miy}\ref{it:matsuoofuniquetype} that $x^{\tau(t)} = x$. Therefore $m^{\mu(t)} \in (M_1^x)^{\mu(t)} = M_1^{x^{\tau(t)}} = M_1^x$. Thus $m_x = \sum_{\mu \in U_x} m^\mu \in M_1^x$. For $y \in \mP$, let $t \in \U(M_\alpha(\mG),\mP)$ such that $(t_x)^t = t_{x^{\tau(t)}}= t_y$ and thus $x^{\tau(t)} = y$. Now, $m_y = m_x^{\mu(t)} \in (M_1^x)^{\mu(t)} = M_1^{x^{\tau(t)}} = M_1^y$.
		\item Because $(t_x)^{t^\prime} = t_y$, $x^{\tau(t^\prime)} = y$. Since $(t_x)^{t^\prime t} = t_{x^{\tau(t^\prime t)}} = t_{y^{\tau(t)}}$, it follows, by (\ref{eq:defmy}), that $m_{y^{\tau(t)}} = (m_x)^{\mu(t^\prime t)} = (m_y)^{\mu(t)}$. \qedhere
	\end{enumerate}
\end{proof}

\begin{thm} \label{thm:1eigen1}
	Consider the setting of Definition \ref{def:1eigenspace}. The subspace $\langle m_y \mid y \in \mP \rangle$ is a submodule of $M$ and the map, defined by
	\[
		M_\alpha(\mG) \to \langle m_y \mid y \in \mP \rangle : y \mapsto m_y,
	\]
	is a surjective homomorphism of $M_\alpha(\mG)$\dash modules.
    In particular, $\langle m_y \mid y \in \mP \rangle$ is a quotient of the regular module for $M_\alpha(\mG)$.
\end{thm}

\begin{proof}
	It suffices to prove that for idempotents $y,z \in \mP$
	\begin{alignat}{2}
		&m_z \cdot z = m_z; \label{1} \\
		&m_y \cdot z = 0 &&\text{ if $y \nsim z$;} \label{2} \\
		&(m_y + m_{y \wedge z} - \alpha m_z) \cdot z = 0 &&\text{ if $y \sim z$;} \label{3} \\
		&(m_y - m_{y \wedge z}) \cdot z = \alpha(m_y - m_{y \wedge z}) &&\text{ if $y \sim z$.} \label{4}
	\end{alignat}
	Statement (\ref{1}) follows from Lemma \ref{lem:1eigen1}\ref{lem:1eigen1:v}. For (\ref{2}), let $z,y \in \mP$ and $y \nsim z$. This implies that $z \in (M_\alpha(\mG))_0^y$. By Lemma \ref{lem:1eigen1}\ref{lem:1eigen1:v}, $m_y \in M_1^y$. From the fusion rule $M_1^y \cdot (M_\alpha(\mG))_0^y \subseteq \{0\}$ we infer $m_y \cdot z = 0$. For (\ref{3}) and (\ref{4}), let $z,y \in \mP$ and $y \sim z$. Since $(m_y - m_{y \wedge z})^{\mu_z}=m_{y \wedge z} - m_y$ by Lemma \ref{lem:1eigen1}\ref{lem:1eigen1:vi}, (\ref{4}) follows. For (\ref{3}), note that \[m_z \cdot (y + z \wedge y - \alpha z) = 0\] since $m_z \in M_1^z$, $y + z \wedge y - \alpha z \in (M_\alpha(\mG))_0^z$ and $M_1^z \cdot (M_\alpha(\mG))_0^z \subseteq \{0\}$. Interchanging the roles of $y$, $z$ and $y\wedge z$ in this relation and relation (\ref{4}) gives us the following 5 relations:
	\begin{align}
		&m_y \cdot (z + y\wedge z - \alpha y) = 0, \label{5}\\
		&m_z \cdot (y \wedge z + y - \alpha z) = 0, \label{6}\\
		&m_{y \wedge z} \cdot (y + z - \alpha y\wedge z) = 0, \label{7}\\
		&(m_z - m_{y \wedge z}) \cdot y - \alpha(m_z - m_{y \wedge z} = 0, \label{8}\\
		&(m_y - m_z) \cdot y \wedge z - \alpha(m_y - m_z) = 0. \label{9}
	\end{align}
	It is easy to verify, using relation (\ref{1}), that $(\ref{5}) - (\ref{6}) + (\ref{7}) + (\ref{8}) - (\ref{9})$ gives us exactly (\ref{3}). \qedhere
\end{proof}

\begin{rem}
	We do not know whether it is always possible to choose $m \in M_1^x$ such that $m_x \neq 0$.
\end{rem}

Theorem~\ref{thm:1eigen2} will tell us when the map from Theorem \ref{thm:1eigen1} is an isomorphism. The proof makes use of a Frobenius form for Matsuo algebras defined in Lemma \ref{lem:1eigen2} below. The only obstructions are the facts that $m_x$ might be zero or that the Frobenius form might be degenerate.

\begin{lem}[{\cite[Corollary 7.4]{HRS15}}] \label{lem:1eigen2}
	Let $\mG$ be a Fischer space. The Matsuo algebra $M_\alpha(\mG)$ is a Frobenius axial algebra. When $\mG$ is connected, its Frobenius form is, up to scalar, uniquely determined. It is given by
	\[
		\langle x,y \rangle = \left\{ \begin{array}{ll}
 			1 & \text{if $x=y$,} \\
 			\frac{\alpha}{2} & \text{if $x \sim y$,} \\
 			0 & \text{if $x \nsim y$,}
 		\end{array} \right.
	\]
	for all $x,y \in \mP$.
\end{lem}

\begin{thm} \label{thm:1eigen2}
	Let $\mG=(\mP,\mL)$ be a connected Fischer space without isolated points. Let $M$ be an $M_\alpha(\mG)$\dash module, $x \in \mP$, $m \in M_1^x \setminus \{ 0 \}$ and define
	\[
		m_x = \sum_{\mu \in U_x} m^\mu.
	\]
	Let $U = \langle \mu_x \mid x \in \mP \rangle \leq \GL(M)$. If $m_x \neq 0$ and the Frobenius form of $M_\alpha(\mG)$ is non-degenerate,
    then $\langle (m_x)^\mu \mid \mu \in U\rangle$ is a regular submodule of $M_\alpha(\mG)$, i.e., it is isomorphic to $M_\alpha(\mG)$ as $M_\alpha(\mG)$\dash module.
\end{thm}

\begin{proof}
	It suffices to prove that the map from Theorem \ref{thm:1eigen1} is injective. Let $m_a$ be the image of $a \in M_\alpha(\mG)$ under this map. We shall prove that for the stated conditions, $m_a\neq 0$ when $a \neq 0$. Since we require $\mG$ to be connected, for every $y \in \mP$, there exists a $\mu \in U$ such that $m_y = (m_x)^\mu$. Because $m_x \neq 0$ and $\mu \in \GL(M)$, $m_y \neq 0$ for all $y \in \mP$.

	Suppose that $a \in M_\alpha(\mG)$ and $m_a = 0$. Let $y \in \mP$ be an arbitrary axis and write $a=a_1 + a_0 + a_\alpha$ where $a_\phi \in (M_\alpha(\mG))^y_\phi$. Then
	\[
		m_a = m_{a_1} + m_{a_0} + m_{a_\alpha} = 0.
	\]
	Since $0 = m_a \cdot y = m_{a \cdot y}$,
	\[
		m_{a_1} + \alpha m_{a_\alpha} = 0.
	\]
	We also have $0 = m_a^{\mu_y} = m_{a^{\tau_y}}$ and therefore
	\[
		m_{a_1} + m_{a_0} - m_{a_\alpha} = 0.
	\]
	From these three equations it follows that $m_{a_1} = m_{a_0} = m_{a_\alpha} = 0$.

	Since the 1\dash eigenspace of $y$ in $M_\alpha(\mG)$ is spanned by $y$, $a_1 = \lambda y$ for some $\lambda \in k$. Because $m_y \neq 0$, we conclude that $a_1$ must be zero. As our choice of $y \in \mP$ was arbitrary, we infer that for every axis in $\mP$ the component of $a$ in its 1\dash eigenspace is zero.

	Let $\langle \; , \; \rangle$ be the Frobenius form for $M_\alpha(\mG)$ as given by Lemma \ref{lem:1eigen2}. Since eigenvectors corresponding to different eigenvalues are perpendicular to each other, our previous conclusion is equivalent to
	\[
		\langle a , y \rangle = 0 \text{ for all $y \in \mP$}.
	\]
	The elements of $\mP$ span $M_\alpha(\mG)$ and therefore $\langle a , n \rangle = 0$ for all $n \in M_\alpha(\mG)$. Because the bilinear form $\langle \; , \; \rangle$ is non-degenerate, $a=0$.
\end{proof}

\begin{rem}
	\begin{enumerate}[(i)]
		\item Let $\mG$ be a connected Fischer space without isolated points and let $M$ be a finite dimensional Frobenius $M_\alpha(\mG)$\dash module. Suppose that the Frobenius form of $M_\alpha(\mG)$ is non-degenerate. Combining Theorem \ref{thm:1eigen2}, Lemma \ref{lem:1eigen2} and Proposition \ref{prop:Maschke}, we can decompose $M$ as a direct sum of regular modules and a module $M^\prime$ for which $\sum_{\mu \in U_x} m^\mu=0$ for every $m \in (M^\prime)_1^x$ and every $x \in \mP$.
		\item Suppose $\mG$ is a finite connected Fischer space without isolated points. Let $A$ be its collinearity matrix. The condition that the Frobenius form from Lemma \ref{lem:1eigen2} is non-degenerate can be expressed as $\det\left(I + \frac{\alpha}{2}A \right) \neq 0$. From this, it is clear that this can only fail for a finite number of choices for $\alpha$.
		\item Since the Miyamoto group of $M_\alpha(\mG)$ acts transitively on the points of a connected Fischer space $\mG$, the number of lines through a point is a constant $d \in \mathbb{N}$ if $\mG$ is finite. The number of points collinear with any given point is then $2d$. The vector of all ones is therefore a $(1+ \alpha d)$\dash eigenvector of $I + \frac{\alpha}{2}A$ where $A$ is the collinearity matrix of $\mG$. It is easy to verify that
	\[
		 \left( \sum_{x \in \mP} x \right) a = (1 + \alpha d) a
	\]
	for all $a \in M_\alpha(\mG)$. If $1 +\alpha d = 0$ and hence $\det\left(I + \frac{\alpha}{2}A\right) =0$, then $\langle \sum_{x \in \mP} x \rangle$ is a 1\dash dimensional $M_\alpha(\mG)$\dash submodule of $M_\alpha(\mG)$. The quotient module is then an $M_\alpha(\mG)$\dash module non-isomorphic to $M_\alpha(\mG)$ with a non-trivial 1\dash eigenspace for every axis $x \in \mP$. The condition that the Frobenius form must be non-degenerate can therefore not be omitted. Note that if $1 +\alpha d \neq 0$, then $M_\alpha(\mG)$ is a unital algebra with unit $(1 + \alpha d)^{-1} \left( \sum_{x \in \mP} x \right)$.
	\end{enumerate}
\end{rem}

\small

\bibliographystyle{alpha}
\bibliography{Articles,Books}

\end{document}